\documentclass{amsart}
\usepackage[english]{babel}
\usepackage[T1]{fontenc}
\usepackage[color,matrix,arrow]{xy}
\usepackage{tikz}
\usetikzlibrary{decorations.markings,arrows,automata,arrows,backgrounds}
\usepackage{amsmath,amssymb,amsthm,amsfonts,amscd,mathrsfs,tikz-cd}
\usepackage{listings}
\usepackage{color}
\usepackage{algpseudocode}
\usepackage{comment}

\usepackage{amsmath}
\usepackage{amssymb}
\usepackage{amsfonts}
\usepackage{comment}
\usepackage{graphicx}
\usepackage[colorinlistoftodos]{todonotes}
\usepackage[colorlinks=true, allcolors=blue]{hyperref}
\usepackage{float}
\usepackage{mathtools}
\usepackage{amssymb,amsmath}
\usepackage{tikz}
\usetikzlibrary{shapes.geometric}
\usepackage{amsthm}

\usepackage{mathtools}

\DeclarePairedDelimiter\floor{\lfloor}{\rfloor}

\theoremstyle{definition}

\newtheorem{lemma}{Lemma}[section]
\newtheorem{definition}{Definition}[section]

\newcommand{\M}    {\mathcal{M}}

\newcommand{\SC}    {\mathrm{SC}}

\newtheorem{thm}{Theorem}

\newcommand{\midarrow}{\tikz \draw[-triangle 90] (0,0) -- +(.1,0);}

\newtheorem{lem}[thm]   {Lemma}
\newtheorem{cor}[thm]   {Corollary}

\newtheorem{defn}[thm]  {Definition}

\newtheorem{prop}[thm]  {Proposition}

\newtheorem{openquestion}{Open Question}

\newcounter{foo}  \Alph{foo}

\title{Star clusters in the Matching, Morse, and Generalized Morse complex}
\author{Connor Donovan, Nicholas A. Scoville}

\address{Department of Mathematics and Computer Science, Ursinus College, Collegeville PA 19426}

\email{codonovan@ursinus.edu}
\email{nscoville@ursinus.edu}

\date{\today}
\keywords{Discrete Morse Theory, complex of discrete Morse functions, Matching Complex, Star Cluster, Cluster Lemma}
\subjclass[2020]{ (Primary) 57Q70;  (Secondary) 55P10, 55U10, 05C70}

\begin{document}

\maketitle

\begin{abstract}
    In this paper, we determine the homotopy type of the Morse complex and matching complex of multiple families of complexes by utilizing star cluster collapses and the Cluster Lemma. We compute the homotopy type of the Morse complex of an extended notion of a star graph, as well as the homotopy type of the matching complex of a Dutch windmill graph. Additionally, we provide alternate computations of the homotopy type of the Morse complex of paths, the homotopy type of the matching complex of paths, and the homotopy type of the matching complex of cycles. We then use this same method of computing homotopy types to investigate the relationship between the homotopy type of the matching complex and the generalized Morse complex.
\end{abstract}

\section{Introduction}

    Let $K$ be a simplicial complex.  In \cite{Barmak13}, Barmak introduced the star cluster of a simplex in $K$ and proved that if $K$ is flag, then the star cluster is contractible. This provided a tool for studying the topology of independence complexes of graphs. For example, he used star clusters to show that the independence complex of any triangle-free graph has the homotopy type of a suspension and that the independence complex of a forest is either contractible or homotopy equivalent to a sphere. K. Iriye used star clusters to construct a matching tree for the independence complex of square grids with cyclic identification \cite{IRIYE-2012}, and S. Goyal et al. used star clusters to compute the homotopy type of the independence complexes of generalised Mycielskian of complete graphs \cite{GOYAL-2021}.  Another important tool in combinatorial topology is the Cluster lemma.  This result was arrived at independently by both Jonsson \cite[Lemma 4.2]{Jonsson2008} and Hersh  \cite[Lemma 4.1]{Hersh2005}, and it is an extremely convenient yet simple way to put a gradient vector field on a complex by gluing together gradient vector fields on a decomposition of the complex.

    A goal of this paper will be to utilize both star clusters and the Cluster Lemma to study the the homotopy type of the Morse complex, generalized Morse complex, and matching complex. The Morse complex of $K$, denoted $\M(K)$, is the simplicial complex of all gradient vector fields on $K$ (see Definition \ref{defn: gvf}).  Like the matching or independence complex of a graph, the Morse complex of a simplicial complex $K$ is a complex that stores certain combinatorial information about $K$, and determining its homotopy type is an interesting question.  Although the Morse complex of $K$ in general is not a flag complex, it is a flag complex when $K=T$ is a tree.  In this special case, we use star clusters and the Cluster lemma to show in Proposition \ref{prop: tree suspension} that $\M(T)$ has the homotopy type of a suspension. We also compute the homotopy type of the Morse complex on any number of paths of two different lengths joined at a single point (Theorem \ref{thm: extended star}) and provide an alternate computation of the homotopy type of the Morse complex of a path (Proposition \ref{prop: path homotopy}) originally computed by D. Kozlov in \cite{Kozlov99}.

     We next investigate the homotopy type of the generalized Morse complex, first introduced in \cite{scoville2020higher}. From the perspective of star clusters, the generalized Morse complex has the advantage that it includes cycles and hence is a flag complex. We compute the homotopy type of the generalized Morse complex of a cycle (Theorem \ref{thm: generalized cycle}) and show that the Morse complex and Generalized Morse complex of a cycle with a single leaf have the same homotopy type.  There seems to be further connections between the Morse complex and Generalized Morse complex, and we discuss some of these possibilities and open questions in the last section.

    Another goal of this paper will be to utilize the idea of the star cluster and the Cluster Lemma to compute the homotopy type of the matching complex of certain complexes. If we consider a graph $G$, the matching complex, denoted $\mathrm{M}(G)$, is the complex constructed from all independent edge sets on $G$. It is easy to see that this is a flag complex.  An interesting connection was made between the Morse and matching complex in \cite{DCNY-2022}\footnote{Bceause they are defining a matching complex on all simplicial complexes, the authors define the matching complex of a graph $G$ to be comprised of elements of matchings on the order poset of $G$, i.e.,  a matching on the barycentric subdivision of $G$}, where the authors provide a natural filtration on the matching complex of a finite simplicial complex that allows us to relate the matching complex to the Morse complex. The authors observe that there is a one-to-one correspondence between elements in the generalized Morse complex of a graph $G$ and matchings on the barycentric subdivision of $G$ so that $\mathcal{G}\M(G)\cong \mathrm{M}(\mathrm{sd}(G))$. Combining this with the fact mentioned above that the Morse complex of a tree is flag, we have the relationship  $\M(T)\cong \mathcal{G}\M(T)\cong \mathrm{M}(\mathrm{sd}(T))$.

    Additionally, the homotopy type of the matching complexes for paths and cycles \cite{Kozlov99} and for forests \cite{MM-2008} have been shown to be either contractible, a sphere, or a wedge of spheres. We provide alternate proofs of the computation of the homotopy type of the matching complex for paths and cycles (Proposition \ref{prop: matching path}, Proposition \ref{prop: matching cycle}) as well as provide a computation of the homotopy type of Dutch windmill graphs (Theorem \ref{thm: matching windmill graph}).

\section{Preliminaries}

In this section, we establish the notation, terminology, and background results that will be needed throughout this paper. All simplicial complexes are assumed to be connected unless otherwise stated. We use $\simeq$ to denote a homotopy equivalence and $\cong$ to denote an isomorphism.

\subsection{Background}
Because we will be taking constructions on graphs, we adopt some graph theoretic language.


\begin{defn}
	A simplicial complex $G$ such that $\dim(G) = 1$ is called a \textbf{graph}.  If $G$ is an acyclic graph, then we call $G$ a \textbf{tree}. The number of edges of a vertex is the \textbf{degree} of the vertex.  A \textbf{leaf} is any vertex of degree $1$. The \textbf{path} $P_n$ on $n$ vertices is the simplicial complex with facets
 $$\{v_0,v_1\}, \{v_1,v_2\}, \cdots \{v_{n-2}, v_{n-1}\},
  $$
 The \textbf{length} of the path $P_n$ is the number of edges ($n-1$) in the path.  A \textbf{cycle} of length $n\geq 3$ is the simplicial complex $C_n$ with facets

 $$\{v_0,v_1\}, \{v_1,v_2\}, \cdots \{v_{n-2}, v_{n-1}\}, \{v_{n-1}, v_0\}.$$
\end{defn}

    \begin{defn}
        A simplicial complex $K$ is a \textbf{flag complex} if for each non-empty set of vertices $\sigma$ such that $\{v_i,v_j\} \in K$ for every $v_i, v_j \in \sigma$, we have that $\sigma \in K$.
    \end{defn}

    The flag complex of a graph $G$ is the smallest flag complex that has $G$ as a 1-skeleton.

    \begin{defn}
	    Let $K$ be a simplicial complex and $v \in K$ be a vertex. The \textbf{\textit{star}} of $v$ in $K$, denoted by $\mathrm{st}(v)$, is the simplicial complex induced by the set of all simplices of $K$ containing $v$. More generally,  the star of a simplex $s$ is the set of simplices having $s$ as a face.
	
	\end{defn}

    \begin{defn} \cite[Definition 3.1]{Barmak13}
        Let $\sigma$ be a simplex of a simplicial complex, $K$. We define the \textbf{star cluster} of $\sigma$ in $K$ as the subcomplex
        \begin{align*}
            \mathrm{SC}_K(\sigma) = \bigcup_{v \in \sigma}{\mathrm{st}_K(v)}
        \end{align*}
        We denote the star cluster of $\sigma$ by $\mathrm{SC}(\sigma)$ when the context is clear.
    \end{defn}

A simple but fundamental property of the star cluster of a simplex is that it is collapsible if the complex is flag.

    \begin{prop}\label{prop: star cluster collapsible}\cite[Lemma 3.2]{Barmak13} The star cluster of a simplex in a flag complex is collapsible.
    \end{prop}

Proposition \ref{prop: star cluster collapsible} is one of the main tools we use in this paper.  The other tool is the following Lemma:

    \begin{lem}\label{lem: cluster lemma}(\cite[Lemma 4.2]{Jonsson2008} and \cite[Lemma 4.1]{Hersh2005}) [Cluster Lemma]
	   Let $\Delta$ be a simplicial complex which decomposes into collections $\Delta_{\sigma}$ of simplices, indexed by the elements $\sigma$ in a partial order $P$ which has a unique minimal element $\sigma_0 = \Delta_0$, Furthermore, assume that this decomposition is as follows:
	        \begin{enumerate}
	            \item Each simplex belongs to exactly one $\Delta_{\sigma}$.
	
	            \item For each $\sigma \in P$, $\bigcup_{\tau \leq \sigma}{\Delta_{\tau}}$ is a subsimplicial complex of $\Delta$.
	        \end{enumerate}
	   For each $\sigma \in P$, let $M_{\sigma}$ be an acyclic matching in $\Delta_{\sigma}$. Then $\bigcup_{\sigma \in P}{M_{\sigma}}$ is an acyclic matching on $K$.
    \end{lem}

    Lemma \ref{lem: cluster lemma} provides a way to put an acyclic matching on the entire complex by patching together acyclic matchings on parts of the complex. The key information is what is left unmatched or the critical simplices.  In some cases, certain collections of critical simplices will uniquely determine the homotopy type of the original complex. This is given in the classical result of Forman.

    \begin{thm}\label{thm: Forman}\cite[Corollary 3.5]{F-95} Let $K$ be a simplicial complex and $M$ an acyclic matching on $K$ with $m_i$ critical simplices of dimension $i$.  Then $K$ has the homotopy type of a CW complex with exactly $m_i$ cells of dimension $i$. In particular, if $m_0=1, m_n=k$, and $m_j=0$ for all $j\neq 0,n$, then $K$ has the homotopy type of a $k$-fold wedge of $S^n.$
    \end{thm}

    All of our computations below will in fact satisfy the stated special case of Theorem \ref{thm: Forman} and thus allow us to determine the homotopy type of the complex in question.

\section{Homotopy Type of the Morse Complex}

In order to describe the Morse complex and Generalized Morse complex, we will need the following.

\begin{defn}
        Let $K$ be a simplicial complex.  A \textbf{discrete vector field} $V$ on $K$ is defined by
        $$V:=\{(\sigma^{(p)}, \tau^{(p+1)}) : \sigma< \tau, \text{ each simplex of } K \text{ in at most one pair}\}.
        $$
        Any pair in $(\sigma,\tau)\in V$ is called a \textbf{regular pair}, and $\sigma, \tau$ are called \textbf{regular simplices} or just \textbf{regular}.  If $(\sigma^{(p)},\tau^{(p+1)})\in V$, we say that $p+1$ is the \textbf{index} of the regular pair. Any simplex in $K$ which is not in $V$ is called \textbf{critical}.
    \end{defn}

    \begin{defn}\label{defn: gvf}
        Let $V$ be a discrete vector field on a simplicial complex $K$.  A \textbf{$V$-path} or \textbf{gradient path}  is a sequence of simplices $$\alpha^{(p)}_0, \beta^{(p+1)}_0, \alpha^{(p)}_1, \beta^{(p+1)}_1, \alpha^{(p)}_2\ldots , \beta^{(p+1)}_{k-1}, \alpha^{(p)}_{k}$$ of $K$ such that $(\alpha^{(p)}_i,\beta^{(p+1)}_i)\in V$ and $\beta^{(p+1)}_i>\alpha_{i+1}^{(p)}\neq \alpha_{i}^{(p)}$ for $0\leq i\leq k-1$. If $k\neq 0$, then the $V$-path is called  \textbf{non-trivial.}  A $V$-path is said to be  \textbf{closed} if $\alpha_{k}^{(p)}=\alpha_0^{(p)}$.  A discrete vector field $V$ which contains no  non-trivial closed $V$-paths is called a \textbf{gradient vector field}.
    \end{defn}

    If the gradient vector field consists of only a single element, we say it is a \textbf{primitive} gradient vector field. We often denote a primitive gradient vector field $\{ (u, uv) \}$ with $p=0$ by $(u)v$.

 \begin{defn}\label{MorseComplexDef2}
        The \textbf{Morse complex} of $K$, denoted $\mathcal{M}(K)$, is the simplicial complex whose vertices are given by primitive gradient vector fields and whose $n$-simplices are given by gradient vector fields with $n+1$ regular pairs.  A gradient vector field $f$ is then associated with all primitive gradient vector fields $f:=\{f_0, \ldots, f_n\}$ with $f_i\leq f$ for all $0\leq i\leq n$.
    \end{defn}

\begin{lemma} \label{lem: flag}
    The Morse complex $\M(K)$ is a flag complex if and only if $K$ is a tree.
\end{lemma}

\begin{proof}
    Let $T$ be a tree and $\M(T)$ the Morse complex of $T$. For $\M(T)$ to be a flag complex, each non-empty set of mutually compatible vertices needs to be all together compatible. In other words, for each non-empty set of vertices $\sigma$ such that $\{v, w\} \subseteq \M(T)$ for every $v, w \in \sigma$, we have that $\sigma \in \M(T)$. Now the only case when a collection of pairwise compatible primitive gradient vector fields may not be compatible is when they form a cycle. But since trees are acyclic, a collection of pairwise compatible primitive gradient vector fields can never form a cycle so that $\M(T)$ is a flag complex.

   Now suppose $\M(K)$ is a flag complex. Clearly neither $K$ nor the 1-skeleton of $K$ can contain a cycle since otherwise there would exist a collection of mutually compatible vertices on $\M(K)$ that are not all together compatible. Thus $K$ must be a tree.
\end{proof}

Although the flag condition greatly reduces the kind of Morse complexes that we can study directly using star clusters, the following result of Barmak will allow us to say something general about the Morse complex of all trees.

\begin{lem}\label{lem: two contractible suspension}\cite[Lemma 3.4]{Barmak13} Let $K$ be a simplicial complex and $K_1, K_2$ be two collapsible subcomplexes such that $K=K_1\cup K_2$.  Then $K$ is homotopy equivalent to $\Sigma(K_1\cap K_2)$.
\end{lem}

We can now show that the Morse complex of all trees is a suspension.

\begin{prop}\label{prop: tree suspension}
    Let $T$ be a tree.  Then $\M(T)$ has the homotopy type of a suspension.
\end{prop}

\begin{proof} We apply Lemma \ref{lem: two contractible suspension} by constructing two collapsible subcomplexes of $\M(T)$ whose union is all of $\M(T)$.
Pick any leaf $\{v_0, v_0v_1\}$ of $T$ and consider the maximum gradient vector field $\sigma_0$ rooted in $v_0$ and the maximum gradient vector field rooted in $v_1$ \cite[Proposition 3.3]{RandScoville20}. These corresponds to simplices $\sigma_0, \sigma_1 \in \M(T)$, respectfully. Define $\M_1(T) = \mathrm{SC}_{\M(T)}(\sigma_0)$ and $\M_2(T) = \SC_{\M(T)}(\sigma_1)$. Then $\M_1(T)$ and $\M_2(T)$ are collapsible subcomplexes of $\M(T)$ by Lemma \ref{lem: flag}.  Furthermore, it is easy to see that $\M(T) = \M_{1}(T) \cup \M_{2}(T)$. Thus $\M(T) \simeq \Sigma(\M_{1}(T) \cap \M_{2}(T))$.

\end{proof}

In addition to the general structure of the Morse complex of a tree, we can use Proposition \ref{prop: star cluster collapsible} and  Lemma \ref{lem: cluster lemma} to compute the homotopy type of some specific classes of trees.  Our first computation is the homotopy type of the Morse complex of a path.  This was originally computed by Kozlov in \cite{Kozlov99}.  Here we provide an alternate computation in a first illustration of our technique.



\begin{prop} \label{prop: path homotopy}
    Let $P_t$ be the path on $t$ vertices, $t\geq 3$.  Then
    \begin{gather*}
       \mathcal{M}(P_t) \simeq  {\begin{cases}
                                *  &\text{if } t = 3n \\
                                \mathbb{S}^{2n-1} &\text{if } t = 3n+1\\
                                \mathbb{S}^{2n} &\text{if } t = 3n+2 \\
                                \end{cases}
                                }
    \end{gather*}
\end{prop}

\begin{proof}
    We apply the Cluster Lemma. In order to do so, we decompose $\M(P_t)$ into collections $\Delta_k$. First, we construct collections of sub-simplices $\sigma_i$ for $i = 0, \ldots n$. We construct collections as follows:
    \begin{enumerate}
        \item Let $\sigma_0 := \mathrm{SC}((v_0)v_1, (v_1)v_2, \ldots, (v_{t-3})v_{t-2}, (v_{t-2})v_{t-1})$
        \item For $1 \leq j \leq n$, we define the following:
            \begin{enumerate}
                \item When $j=2k-1$, let $\sigma_j := \mathrm{st}((v_{t-(3k-1)})v_{t-3k})$
                \item When $j=2k$, let $\sigma_j := \mathrm{st}((v_{3k})v_{3k-1})$
            \end{enumerate}
        \item Let $\sigma_{n+1} := \M(P_t) - \bigcup_{i=0}^{n}{\sigma_{i}}$
    \end{enumerate}
    Now define $\Delta_0 := \sigma_0$ and $\Delta_k := \sigma_k - \bigcup_{j=0}^{k-1}{\sigma_j}$, and observe that $\bigcup_{k=0}^{n+1}{\Delta_k} = \M(P_t)$. Define an acyclic matching on each $\Delta_j$ as follows:

    We know that $\Delta_0$ is collapsible by Proposition \ref{prop: star cluster collapsible} and Lemma \ref{lem: flag} so $\Delta_0$ has an acylcic matching with a single unmatched $0$-simplex.

    Let $j=2k-1$, $1 \leq j \leq n$.  Any simplex $V\in \Delta_j$ by definition contains $(v_{t-(3k-1)})v_{t-1}$. Match $V$ with $V\cup \{(v_{t-(3k-2)})v_{t-(3k-1)}\}$ (or $V-\{(v_{t-(3k-2)})v_{t-(3k-1)}\}$ if $V$ already contains this vector). In this way, all simplices in $\Delta_{2k-1}$ are matched with no unmatched simplices. Furthermore, since this matching is a subset of the matching on the cone on $(v_{t-(3k-1)})v_{t-1}$, it is acyclic.

    Let $j=2k$, $2 \leq j \leq n$. Any simplex $V\in \Delta_j$ by definition contains $(v_{3k})v_{3k-1}.$ Match $V$ with $V\cup \{(v_{3k-1})v_{3k-2}\}$ (or $V$ with this vector removed, as above). In this way, all simplices in $\Delta_{2k}$ are matched with no unmatched simplices. Again, this matching is a subset of the matching on a cone so it is acyclic.

    Now consider $\Delta_{n+1}$. We have three cases:

    (Note: When considering $n=1$ in cases 2 and 3, disregard matchings containing vertices with negative indices e.g. $v_{-1}$)
    \begin{description}
    \item[Case 1]
        Let $t=3n$. Then $\Delta_{n+1} = \emptyset$, and thus $\M({P_t}) \simeq *$.


    \item[Case 2]
        Let $t=3n+1$. Then $\Delta_{n+1}$ contains a single simplex $V$ of dimension $(2n-1)$ satisfying
        \begin{align*}
           (v_{3\floor{\frac{n}{2}}})v_{3\floor{\frac{n}{2}}-1}, (v_{3\floor{\frac{n}{2}}+2})v_{3\floor{\frac{n}{2}}+1} \not\in V.
        \end{align*}
         Thus by Theorem \ref{thm: Forman}, $\M({P_t}) \simeq \mathbb{S}^{2n-1}$.




    \item[Case 3]

        Let $t=3n+2$.  Then $\Delta_{n+1}$ contains a single simplex $V$ of dimension $2n$ satisfying
 \begin{align*}
           (v_{3\floor{\frac{n}{2}}})v_{3\floor{\frac{n}{2}}-1}, (v_{3\floor{\frac{n}{2}}+3})v_{3\floor{\frac{n}{2}}+2} \not\in V.
        \end{align*}
        Thus by Theorem \ref{thm: Forman}, $\M({P_t}) \simeq \mathbb{S}^{2n}$.




     \end{description}
\end{proof}

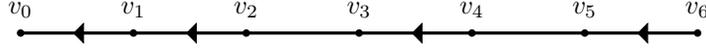
\begin{figure} [h]

\begin{tikzpicture}[scale=1.5, node distance = {20mm}, thick, main/.style = {draw, circle}]
\begin{scope}[very thick, every node/.style={sloped,allow upside down}]
		
            \node[label={$v_{0}$}, inner sep=1pt, circle, fill=black](v3) at (-2,0) {};
            \node[label={$v_{1}$}, inner sep=1pt, circle, fill=black](v4) at (-1,0) {};
            \node[label={$v_{2}$}, inner sep=1pt, circle, fill=black](v5) at (0,0) {};
            \node[label={$v_{3}$}, inner sep=1pt, circle, fill=black](v6) at (1,0) {};
            \node[label={$v_{4}$}, inner sep=1pt, circle, fill=black](v7) at (2,0) {};
            \node[label={$v_{5}$}, inner sep=1pt, circle, fill=black](v7) at (3,0) {};
            \node[label={$v_{6}$}, inner sep=1pt, circle, fill=black](v7) at (4,0) {};

		    \draw (-1,0) -- node {\midarrow} (-2,0);
		    \draw (0,0) -- node {\midarrow} (-1,0);
		    \draw (0,0) -- (1,0);
		    \draw (2,0) -- node {\midarrow} (1,0);
		    \draw (2,0) -- (3,0);
		    \draw (4,0) -- node {\midarrow} (3,0);
\end{scope}
\end{tikzpicture}

\caption{For $t=7=3(2)+1$, Case 2 of the proof of Theorem \ref{prop: path homotopy} implies that $V\in \Delta_3$ will result in the gradient vector field (critical simplex) in $\M(P_7)$ pictured above.} \label{lem paths example}


\end{figure}

Recall that the \textbf{star graph} $S_n$ on $n+1$ vertices is the complete bipartite graph $K_{1,n}$.  Alternatively, we may view $S_n$ as the result of taking $n$ paths of length 1 and gluing them to a common endpoint (the so-called wedge product). We generalize $S_n$ in the following definition.

\begin{defn}
    An \textbf{extended star graph}, denoted $S_{v_1,v_2,v_3}$, is the graph obtained by starting with $v_1$ paths of length 1, $v_2$ paths of length 2, and $v_3$ paths of lengths 3 and identifying an endpoint of each path with a fixed vertex $c$ called the \textbf{center}.
    By an \textbf{extended leaf of length $k$}, we mean a path of length $k$ from the center vertex, $c$, to a vertex, $v_k$, of degree $1$.
\end{defn}

Clearly $S_k=S_{k,0,0}$ recovers the star graph.  It was shown in \cite[Proposition 3.5]{donovan_lin_scoville_2022} that not only is $\M(S_n)$ (strongly) collapsible for $n\geq 2$, but that any complex with at least two leaves sharing a common vertex is strongly collapsible.  Hence, we let $v_0=0$ in our computation below.

\begin{center}
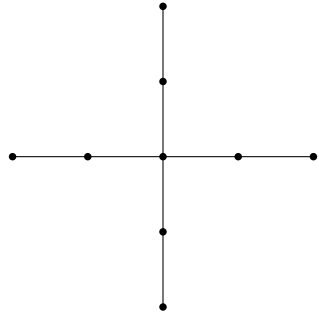
\begin{figure} [h]

\begin{tikzpicture}
    \node[inner sep=1pt, circle, fill=black](v3) at (0,0) {};
    \node[inner sep=1pt, circle, fill=black](v3) at (1,0) {};
    \node[inner sep=1pt, circle, fill=black](v3) at (2,0) {};
    \node[inner sep=1pt, circle, fill=black](v3) at (0,1) {};
    \node[inner sep=1pt, circle, fill=black](v3) at (0,2) {};
    \node[inner sep=1pt, circle, fill=black](v3) at (-1,0) {};
    \node[inner sep=1pt, circle, fill=black](v3) at (-2,0) {};
    \node[inner sep=1pt, circle, fill=black](v3) at (0,-1) {};
    \node[inner sep=1pt, circle, fill=black](v3) at (0,-2) {};

    \draw (0,0) -- (1,0);
    \draw (1,0) -- (2,0);
    \draw (0,0) -- (0,1);
    \draw (0,1) -- (0,2);
    \draw (0,0) -- (-1,0);
    \draw (-1,0) -- (-2,0);
    \draw (0,0) -- (0,-1);
    \draw (0,-1) -- (0,-2);

\end{tikzpicture}

    \caption{The extended star graph, $S_{0,4,0}$. By corollary [\ref{cor: ex star 2}], we see that $\M(S_{0,4}) \simeq \mathbb{S}^{4} \vee \mathbb{S}^{4} \vee \mathbb{S}^{4}$.}

\end{figure}
\end{center}

\begin{thm}  \label{thm: extended star}
    Let $S_{0,n,m}$ be an extended star graph. Then,
    \begin{align*}
        \M(S_{0,n,m}) \simeq \vee^{n-1}\mathbb{S}^{2m+n}
    \end{align*}
\end{thm}

\begin{proof}
    Define a collection of subsimplices $\sigma_i$ for $i = 0, \ldots n$ on $\M(S_{0,n,m})$ as follows:

    Let $c$ be the center vertex of $S_{0,n,m}$ and let $\{v_{a_i}v_{b_i}, v_{b_i}\}$ be the leaf of each extended leaf of length 2, $i=1,2,\ldots, n$, and $\{v_{\alpha_j}v_{\beta_j}, v_{\beta_j}\}$ the leaf of each extended leaf of length 3, $j=1, 2,\ldots, m$ with $v_{\gamma_j} \not= v_{\beta_j}$ the other neighbor of $v_{\alpha_j}$.
    \begin{enumerate}
        \item  Let $\sigma_0$ be the star cluster of the gradient vector field rooted in $c$. Such a gradient vector field exists and is unique by \cite[Proposition 3.3]{RandScoville20}.
        \item Let $\sigma_1 := \cup_{i=1}^{m}\mathrm{st}(\{(c)v_{\gamma_i}\})$
    \end{enumerate}

    Now define $\Delta_0 := \sigma_0, \Delta_1 := \sigma_1 - \sigma_0$, and $\Delta_2:=\M(S_{0,n,m})-(\sigma_0\cup \sigma_1)$. Clearly $\Delta_0\cup\Delta_1\cup\Delta_2= \M(S_{0,n,m})$ so we can apply the Cluster Lemma. We define an acyclic matching on each $\Delta_j$ as follows:

    First, $\Delta_0$ is collapsible by Proposition \ref{prop: star cluster collapsible} and Lemma \ref{lem: flag} so there is an acyclic matching on $\Delta_0$ with a single critical 0-simplex.

    To construct a matching on $\Delta_1$, we first observe that a typical element of $\Delta_1$ is of the form $(c)v_{\gamma_i}$ along with other arrows pointing away from the center vertex $c$. Furthermore, because $\sigma_0$ contains all gradient vector fields with any arrow pointing towards $c$, all elements of $\Delta_1$ are not compatible with any arrow pointing towards $c$. Upon inspection, there are exactly $2m$ such gradient vector fields. Match the $(2m+n-1)$-simplex of $\Delta_1$ containing $(c)v_{\gamma_i}$ but not containing $(v_{\gamma_i})v_{\alpha_i}$ to the corresponding $(2m+n)$-simplex containing both $(c)v_{\gamma_i}$ and $(v_{\gamma_i})v_{\alpha_i}$. This produces an acyclic matching on all elements in $\Delta_1$.

    Lastly, observe that $\Delta_2$ contains $n+1$ elements. We will create a single matching, leaving $n-1$ unmatched $(2m+n)$-simplices and hence critical. A typical element of $\Delta_2$ is of the form $(\bigcup_{i=1}^n(v_{a_i})v_{b_i})\cup (\bigcup_{i=1}^m(v_{\gamma_i})v_{\alpha_i})\cup (\bigcup_{i=1}^m(v_{\alpha_i})v_{\beta_1})$ along with possibly one of $(c)v_{a_i}$. Match the $(2m+n-1)$-simplex of $\Delta_2$ containing none of the $(c)v_{a_i}$ with the $(2m+n)$-simplex containing $(c)v_{a_1}$.

If $n>1$ then, there are $n-1$ unmatched $(2m+n)$-simplices $\tau_i$, where each $\tau_i$ contains $(c)v_{a_i}$ for $i=2,3,\ldots, n-1$.  Thus $\M(S_{0,n,m}) \simeq \vee^{n-1}\mathbb{S}^{2m+n}$.

\end{proof}

We obtain several special cases which we list as corollaries below.

\begin{cor}
    Let $S_{0,1,n}$ be an extended star graph. Then,
    \begin{align*}
        \M(S_{0,1,n}) \simeq *
    \end{align*}
\end{cor}

\begin{cor} \label{cor: ex star 2}
    Let $S_{0,n}$ be an extended star graph. Then,
    \begin{align*}
        \M(S_{0,n}) \simeq \vee^{n-1}\mathbb{S}^{n}
    \end{align*}
\end{cor}

\begin{cor}
    Let $S_{0,0,n}$ be an extended star graph. Then,
    \begin{align*}
        \M(S_{0,0,n}) \simeq \mathbb{S}^{2n-1}
    \end{align*}
\end{cor}

\section{Homotopy Type of the Generalized Morse Complex}\label{sec: Homotopy Type of the Generalized Morse Complex}

The following was defined in \cite{scoville2020higher} in order to estimate the connectivity of the Morse complex.

\begin{defn}
    The \textbf{generalized Morse complex} $\mathcal{G}\M(K)$ of a simplicial complex, $K$, is the simplicial complex whose vertices are the primitive gradient vector fields on $K$, with a finite collection of vertices spanning a simplex whenever the primitive gradient vector fields are pairwise compatible. Reworded, the simplices of $\mathcal{G}\M(K)$ are the discrete vector fields on $K$, with face relation given by inclusion.
\end{defn}

Note that $\mathcal{G}\M(K)$ is a flag complex since it allows closed $V$-paths on $K$. We now compute the homotopy type of the generalized Morse complex for cycles.  First, a definition.

\begin{definition}
   Let $C_t$ be a cycle, $t\geq 3$, with vertices $v_0, \ldots, v_{t-1}$. Let $V_k:=\{(v_{i+1})v_i: k\leq i\leq t-1\}$. We define $\mathrm{st_{mod}}(V_k) := \{ \sigma \in \mathrm{st}(V_k) : (v_{k})v_{k-1} \not \in \sigma\}$.
\end{definition}

\begin{thm} \label{thm: generalized cycle}
 Let $C_t$ be the cycle on $t$ vertices, $t > 3$.  Then
    \begin{gather*}
       \mathcal{G}\M(C_t) \simeq  {\begin{cases}
                                \mathbb{S}^{2n-1} \vee \mathbb{S}^{2n-1} &\text{if } t = 3n \\
                                \mathbb{S}^{2n} &\text{if } t = 3n+1\\
                                \mathbb{S}^{2n} &\text{if } t = 3n+2 \\
                                \end{cases}
                                }
    \end{gather*}
\end{thm}

\begin{proof}
    We decompose $\mathcal{G}\M(C_t)$ into collections $\Delta_k$. We begin by constructing the following collections:
    \begin{enumerate}
        \item Let $\sigma_0 := \mathrm{SC}(\{(v_0)v_1, (v_1)v_2, \ldots , (v_{t-1})v_{0}\})$
        \item For $1 \leq j \leq t-2$, let $\sigma_j := \mathrm{st_{mod}}(\{(v_{j})v_{j-1} (v_{j+2})v_{j+1}\})$

    \end{enumerate}
    Define $\Delta_0 := \sigma_0$ and $\Delta_k := \sigma_k - \bigcup_{j=0}^{k-1}{\sigma_j}$. Then $\bigcup_{k=0}^{t-2}{\Delta_k} = \mathcal{G}\M(C_t)$.  Clearly $\Delta_0$ is collapsible. Now match $k$-simplex of the form $\{(v_{j})v_{j-1} (v_{j+2})v_{j+1} \ldots \}$ with the $(k+1)$-simplex of the form $\{(v_{j})v_{j-1} (v_{j+1})v_j (v_{j+2})v_{j+1} \ldots \}$. There are three cases to consider.

    \begin{description}
    \item[Case 1]
        Let $t=3n$. Then there will be two critical $(2n-1)$-simplices, both of which were excluded from every $\sigma_j$ by the definition of $\mathrm{st_{mod}}$. However, all other simplices have been matched. Thus $\mathcal{G}\M(C_{3n}) \simeq \mathbb{S}^{2n-1} \vee \mathbb{S}^{2n-1}$.

    \item[Case 2]
        Let $t=3n+1$. Then there will be one critical $(2n)$-simplex while all other simplices have been matched. Thus $\mathcal{G}\M(C_{3n+1}) \simeq \mathbb{S}^{2n}$.

    \item[Case 3]
        Let $t=3n+2$. There is a single critical $(2n)$-simplex which was excluded from every $\sigma_j$ by the definition of $\mathrm{st_{mod}}$ with all other simplices matched. Thus $\mathcal{G}\M(C_{3n+2}) \simeq \mathbb{S}^{2n}$.
    \end{description}
\end{proof}




Now we investigate the generalized Morse complex of a cycle with a leaf attached. We use the notation $C_t\vee \ell$ to denote the cycle of length $t$ with a leaf $\ell$ joined to some vertex of $C_t$.

\begin{thm} \label{thm: general of cycle and leaf}
    Let $C_t$ be the path on $t$ vertices, $t > 3$.  Then
    \begin{gather*}
       \mathcal{G}\M(C_t \vee \ell) \simeq  {\begin{cases}
                                * &\text{if } t = 3n \\
                                \mathbb{S}^{2n} &\text{if } t = 3n+1\\
                                \mathbb{S}^{2n+1} &\text{if } t = 3n+2 \\
                                \end{cases}
                                }
    \end{gather*}
\end{thm}

\begin{proof}
    Let $\{v_1,v_0v_1\}$ be the leaf attached to $v_1\in C_t$.

  To apply the Cluster lemma, we first construct collections as follows:
    \begin{enumerate}
        \item Let $\sigma_0 := \mathrm{SC}(\{(v_0)v_1(v_1)v_2(v_2)v_3 \ldots (v_{n})v_{1}\})$
        \item For $1 \leq j \leq n$,
        \begin{enumerate}
            \item Let $j = 2k-1$ and define
                \begin{align*}
                    \sigma_j := \mathrm{st}(\{(v_{1+3(k-1)})v_{0+3(k-1)}(v_{3+3(k-1)})v_{2+3(k-1)}\})
                \end{align*}

            \item Let $j = 2k$ and define
                \begin{align*}
                    \sigma_j := \mathrm{st}(\{(v_{t-3(k-1)})v_{(t-1)-3(k-1)}(v_{(t-2)-3(k-1)})v_{(t-3)-3(k-1)}\})
                \end{align*}
        \end{enumerate}
        \item For $j = n+1$,
        \begin{enumerate}
            \item if $n+1 = 2k-1$, then
                \begin{align*}
                    \sigma_{n+1} := \mathrm{st}(\{(v_{t-3(k-1)})v_{(t-1)-3(k-1)}(v_{(t-1)-3(k-1)})v_{(t-2)-3(k-1)}\})
                \end{align*}
            \item if $n+1 = 2k$, then
                \begin{align*}
                    \sigma_{n+1} := \mathrm{st}(\{(v_{1+3(k-1)})v_{0+3(k-1)}(v_{2+3(k-1}))v_{1+3(k-1)}\})
                \end{align*}
        \end{enumerate}
    \end{enumerate}
    Let $\Delta_0 := \sigma_0$ and $\Delta_k := \sigma_k - \bigcup_{j=0}^{k-1}{\sigma_j}$. Then $\bigcup_{k=0}^{n+1}{\Delta_k} = \mathcal{G}\M(C_t \vee l)$. Clearly $\Delta_0$ is collapsible. Now match each $\Delta_j$ for $1 < j < n$ by the following:

    If $j=2k-1$, match each $m$-simplex of the form
    \begin{align*}
          \{(v_{1+3(k-1)})v_{0+3(k-1)}(v_{3+3(k-1)})v_{2+3(k-1)} \ldots\}
    \end{align*}
    to the corresponding $m+1$-simplex of the form
    \begin{align*}
        \{(v_{1+3(k-1)})v_{0+3(k-1)}(v_{2+3(k-1)})v_{1+3(k-1)}(v_{3+3(k-1)})v_{2+3(k-1)} \ldots\}
    \end{align*}

    If $j=2k$,  match each $m$-simplex of the form
    \begin{align*}
        \{(v_{t-3(k-1)})v_{(t-1)-3(k-1)}(v_{(t-2)-3(k-1)})v_{(t-3)-3(k-1)} \ldots \}
    \end{align*}
    to the corresponding $m+1$-simplex of the form
    \begin{align*}
        \{(v_{t-3(k-1)})v_{(t-1)-3(k-1)}(v_{(t-1)-3(k-1)})v_{(t-2)-3(k-1)}(v_{(t-2)-3(k-1)})v_{(t-3)-3(k-1)} \ldots\}
    \end{align*}
    Thus all simplicies in $\Delta_j$ for $1 < j < n$ have been matched.

    Now we must match simplices in $\Delta_{n+1}$. We consider three cases:
    \begin{description}
    \item[Case 1]
        Let $t=3n$. Then $\Delta_{n+1} = \emptyset$, and thus $\mathcal{G}\M(C_t \vee l) \simeq *$.

    \item[Case 2]
        Let $t=3n+1$. Then $\Delta_{n+1}$ only contains one $2n$-simplex. Thus $\mathcal{G}\M(C_t \vee l) \simeq \mathbb{S}^{2n}$.

    \item[Case 3]
        Let $t=3n+2$. Then $\Delta_{n+1}$ only contains one $2n+1$-simplex. Thus $\mathcal{G}\M(C_t \vee l) \simeq \mathbb{S}^{2n+1}$.
    \end{description}

\end{proof}

The homotopy type of the Morse complex of $C_t\vee \ell$ was computed in \cite[Proposition 5.6]{donovan_lin_scoville_2022}.  It turns out to be the same as the homotopy type of the Generalized Morse complex of $C_t\vee \ell$.  We thus have

\begin{cor}
    Let $C_t \vee \ell$ be a cycle with a leaf. Then,
    \begin{gather*}
        \mathcal{G}\M(C_t \vee \ell) \simeq \M(C_t \vee \ell).
    \end{gather*}

\end{cor}




A collapse of $\mathcal{G}\M(C_t \vee \ell)$ onto $ \M(C_t \vee \ell)$ can be seen by considering the closed V-paths in  $\mathcal{G}\M(C_t \vee \ell)$ that are added to $\M(C_t \vee \ell)$. We see that there are four such V-paths: a clockwise cycle, a counterclockwise cycle, a clockwise cycle with an inward facing arrow on the leaf, and a counterclockwise cycle with an inward facing arrow on the leaf. By matching the clockwise cycle with the clockwise cycle with an inward facing arrow on the leaf and also matching the counterclockwise cycle to the counterclockwise cycle with an inward facing arrow on the leaf, we have collapsed $\mathcal{G}\M(C_t \vee \ell)$ back into $\M(C_t \vee \ell)$, showing a homotopy equivalence.

\section{Homotopy Type of the Matching Complex}

A well-known complex that one associated to a graph is the matching complex.

\begin{defn}
    Let the \textbf{matching complex} of a graph, $G$, denoted $\mathrm{M}(G)$, is a simplicial complex with vertices given by edges of $G$ and faces given by matchings of $G$, where a matching is a subset of edges $H \subseteq E(G)$ such that any vertex $v \in V(H)$ has degree at most $1$.
\end{defn}









The homotopy types of the matching complexes of the path and cycle were computed in \cite{Kozlov99}. As in the case of the Morse complex for a path, we provide an alternate proof of these computations using discrete Morse theory, the cluster lemma, and star clusters. Then we provide a new result, computing the homotopy type of the matching complex for Dutch windmill graphs. We first make the following simple but useful observation. As observed in \cite{DCNY-2022}, $\mathcal{G}\M(G)\cong \mathrm{M}(\mathrm{sd}(G))$ for $G$ any graph.  Thus the results in section \ref{sec: Homotopy Type of the Generalized Morse Complex} hold for the matching complex on the barycentric subdivision of the graph in question.  It was furthermore proved in \cite[Proposition 3.5]{donovan_lin_scoville_2022} that if a graph $G$ has two leaves sharing a common vertex, then the Morse complex is contractible.  The same result holds for the generalized Morse complex.  We thus have the following:

\begin{cor} If a graph $G$ has two leaves sharing a common vertex, then $\mathrm{M}(\mathrm{sd}(G))$ is contractible.
\end{cor}

\begin{prop}  \label{prop: matching path}
    Let $P_t$ be a path on $t \geq 3$ vertices. Then
    \begin{align*}
        \mathrm{M}(P_t) \simeq  {\begin{cases}
                                \mathbb{S}^{n-1}  &\text{if } t = 3n \\
                                \mathbb{S}^{n-1} &\text{if } t = 3n+1\\
                                * &\text{if } t = 3n+2 \\
                                \end{cases}
                                }
    \end{align*}
\end{prop}

\begin{proof}
    We apply the Cluster Lemma. In order to do so, we decompose $\mathrm{M}(P_t)$ into collections $\Delta_k$. First, we construct collections of sub-simplices $\sigma_i$. We construct collections as follows:
    \begin{enumerate}
        \item Let $\sigma_0 := \mathrm{SC}(\{ \bigcup^{k}_{i=0}(v_{3i}v_{3i+1})\})$, $k \leq n$
        \item Let $\sigma_1 := \mathrm{st}{\{(v_1v_2)\}}$
    \end{enumerate}

    Let $\Delta_0 := \sigma_0$ and $\Delta_1 := \sigma_1 - \sigma_0$. Now any maximal matching of $P_t$ contains either $v_0v_1$ or $v_1v_2$.  If it contains $v_0$, then it is in $\Delta_0$.  If it contains $v_1v_2$, then it is in $\Delta_1$. Hence $\Delta_0\cup \Delta_1 = \mathrm{M}(P_t)$  so that we define an acyclic matching on  $\Delta_0, \Delta_1$ and apply the Cluster Lemma.

    Now $\Delta_0$ is flag so it is collapsible by Proposition \ref{prop: star cluster collapsible} and Lemma \ref{lem: flag}. To construct a matching on $\Delta_1$, we consider three cases:
    \begin{description}
    \item[Case 1]
        Let $t=3n$. Then $\Delta_1$ is a single simplex given by $\{ \bigcup^{n-1}_{i=0}(v_{3i+1}v_{3i+2})\}$. Hence this corresponds to an $(n-1)$-simplex in the Morse complex and thus is critical so that $\mathrm{M}(P_{3n}) \simeq \mathbb{S}^{n-1}$.

    \item[Case 2]
            Let $t=3n+1$. As in Case 1, $\Delta_1$ is a single matching given by $\{ \bigcup^{n-1}_{i=0}(v_{3i+1}v_{3i+2})\}$. This matching corresponds to a critical $(n-1)$-simplex in the Morse complex and thus $\mathrm{M}(P_{3n+1}) \simeq \mathbb{S}^{n-1}$.

    \item[Case 3]
        Let $t=3n+2$. Then $\Delta_1 = \emptyset$. Thus $\mathrm{M}(P_{3n+2}) \simeq *$.
    \end{description}
\end{proof}

We also provide an alternate proof for computing the homotopy type of the matching complex of the cycle using the same technique and a similar matching.

\begin{prop} \label{prop: matching cycle}
     Let $C_t$ be a cycle on $t \geq 3$ vertices. Then
    \begin{align*}
        \mathrm{M}(C_t) \simeq  {\begin{cases}
                                \mathbb{S}^{n-1} \vee \mathbb{S}^{n-1}  &\text{if } t = 3n \\
                                \mathbb{S}^{n-1} &\text{if } t = 3n+1\\
                                \mathbb{S}^n &\text{if } t = 3n+2 \\
                                \end{cases}
                                }
    \end{align*}
\end{prop}

\begin{proof}
    As usual, we apply the Cluster Lemma by first constructing collections of subsimplices $\sigma_i$.
    \begin{enumerate}
        \item Let $\sigma_0 := \mathrm{SC}(\{ \bigcup^{k}_{i=0}(v_{3i}v_{3i+1})\})$, $k \leq n$ ($k\leq n-1$ when $t=3n+1$)
        \item Let $\sigma_1 := \mathrm{st}{\{(v_{(t-1)}v_0)\}}$
        \item Let $\sigma_2 := \mathrm{st}{\{(v_1v_2)\}}$
    \end{enumerate}

    Define $\Delta_0 := \sigma_0, \Delta_1 := \sigma_1 - \sigma_0$, and $\Delta_2:= \sigma_2-(\sigma_0\cup \sigma_1)$. Since every matching of $C_t$ is in one of the $\sigma_i$, it follows that $\Delta_0\cup \Delta_1\cup \Delta_2 = \mathrm{M}(C_t)$.  To define an acyclic matching on the $\Delta_i$, we first observe that $\Delta_0$ is collapsible.

    The matchings on both $\Delta_1$ and $\Delta_2$ are considered in three cases:
    \begin{description}
    \item[Case 1]
    Let $t=3n$. In $\Delta_1$, there exists one $(n-1)$-simplex, $\{\bigcup^{k}_{i=0}(v_{2+3i}v_{3+3i})\}$. Thus it cannot be matched so it it critical. In $\Delta_2$, there exists one $(n-1)$-simplex of the form $\{\bigcup^{k}_{i=0}(v_{1+3i}v_{2+3i})\}$ which also cannot be matched. Thus $\mathrm{M}(C_{3n}) \simeq \mathbb{S}^{n-1} \vee \mathbb{S}^{n-1}$.

    \item[Case 2]
        Let $t=3n+1$. Any $(n-1)$-simplex $V$ in $\Delta_1$ does not contain $\{(v_{t-3}v_{t-2})\}$ so we match $V$ with $V\cup \{(v_{t-3}v_{t-2})\}$. This yields a perfect acyclic matching on $\Delta_1$. Now there is only one simplex in $\Delta_2$; namely, the $(n-1)$-simplex $\{ \bigcup^{n-1}_{i=0}(v_{3i+1}v_{3i+2})\}$. This $(n-1)$-simplex is critical, hence $\mathrm{M}(C_{3n+1}) \simeq \mathbb{S}^{n-1}$.



    \item[Case 3]
    Let $t=3n+2$. For each $(n-1)$-simplex $V$ of $\Delta_1$, there is exactly one $k$, $0\leq k \leq n-1$, such that both $v_{3k+1}v_{3k+2}$ and $v_{3k+2}v_{3k+3}$ are not in $V$. Match this $V$ with $V\cup \{v_{3k+2}v_{3k+3}\}$.  Then there is one $n$-simplex left unmatched, namely, $\{ \bigcup^{n}_{i=0}(v_{3i+1}v_{3i+2})\}$. Observe that $\Delta_2$ is empty, and thus $\mathrm{M}(C_{3n+2}) \simeq \mathbb{S}^{n}$.

    \end{description}
\end{proof}

\begin{defn}
    A \textbf{centipede graph}, $\mathscr{C}_t$ is a graph obtained by adding a leaf to each vertex on a path $P_t$. If $v_0, \ldots, v_{t-1}$ are the vertices of $P_t$, denote the vertex of the leaf added to $v_i$ by $v'_i$.
\end{defn}

\begin{prop}
    Let $\mathscr{C}_t$ be a centipede graph. Then
    \begin{gather*}
       \mathrm{M}(\mathscr{C}_t) \simeq  {\begin{cases}
                                \mathbb{S}^{n-1} &\text{if } t = 2n \\
                                * &\text{if } t = 2n+1\\
                                \end{cases}
                                }
    \end{gather*}
\end{prop}

\begin{proof}
    Let $\mathscr{C}_t$ be a centipede graph. We apply the Cluster Lemma and construct collections as follows:
    \begin{enumerate}
        \item Let $\sigma_0 := \mathrm{SC}(\{\bigcup_{i=0}^{t-1}(v_iv_i')\})$
        \item Let $\sigma_1 := \mathrm{st}((v_0v_1))$
    \end{enumerate}
     Define $\Delta_0 := \sigma_0$ and $\Delta_1 := \sigma_1 - \sigma_0$ so that $\Delta_0\cup
     \Delta_1= \mathrm{M}(\mathscr{C}_t)$. Define an acyclic matching on each $\Delta_i$ as follows:

    We know $\Delta_0$ is collapsible by Proposition \ref{prop: star cluster collapsible} and Lemma \ref{lem: flag}.

   For $\Delta_1$, we have two cases:
    \begin{description}
    \item[Case 1] Let $t=2n$. Then the only element in $\Delta_1$ is $\{\bigcup_{j=0}^{n-1}(v_{2j}v_{2j+1})\}$, an $(n-1)$-simplex. Hence $\mathrm{M}(\mathscr{C}_t) \simeq \mathbb{S}^{n-1}$.

    \item[Case 2] Let $t=2n+1$. Then $\Delta_1 = \emptyset$. Thus $\mathrm{M}(\mathscr{C}_t) \simeq *$.

    \end{description}

\end{proof}

\begin{defn}
    Let $D^n_m$ be a \textbf{Dutch windmill graph}. $D^n_m$ is obtained by taking $n$ copies of the cycle $C_m$ and joining them at a common vertex.
\end{defn}

\begin{thm} \label{thm: matching windmill graph}
    Let $D^n_m$ be a Dutch windmill graph. Then
    \begin{gather*}
       \mathrm{M}(D^n_m) \simeq  {\begin{cases}
                                * &\text{if } m = 3k \\
                                \mathbb{S}^{nk-1} &\text{if } m = 3k+1 \\
                                \vee^{2n-1}\mathbb{S}^{nk} &\text{if } m = 3k+2\\
                                \end{cases}
                                }
    \end{gather*}
\end{thm}

\begin{proof}
   Let $D_m^n$ be the Dutch windmill graph with center vertex $v_0$ and for each of the $n$ cycles $C_m$, let $v_{(j)_i}$ denote vertex $j$ of cycle $i$, $0\leq j\leq m-1$ and $1\leq i\leq n.$ We apply the Cluster lemma by defining the following collections:

    \begin{enumerate}
        \item Let $\sigma_0 := \mathrm{SC}\{\bigcup_{i=1}^{n} (\bigcup_{j=0}^{k-1} (v_{({3j+1})_i}v_{({3j+2})_i}))\}$
        \item For $1 \leq \nu \leq k-1$, let $\sigma_{\nu} := \bigcup^{n}_{i=1}(\mathrm{st}(\bigcup^{k-\nu}_{j=1}(v_{(3j)_i}v_{(3j+1)_i})))$
        \item Let $\sigma_{k} := \bigcup^{n}_{i=1}(\mathrm{st}(\bigcup^{k-1}_{j=0}(v_{(3j+2)_i}v_{(3j+3)_i})))$
    \end{enumerate}
    Define $\Delta_0 := \sigma_0$ and $\Delta_{\beta} := \sigma_{\beta} - \bigcup_{\alpha=0}^{\beta-1}{\sigma_{\alpha}}$. Then $\bigcup_{\beta=0}^{k}{\Delta_{\beta}} = \mathrm{M}(D^n_m)$. We now define an acyclic matching on each $\Delta_{\beta}$ as follows:

    We know $\Delta_0$ is collapsible by Proposition \ref{prop: star cluster collapsible} and Lemma \ref{lem: flag}. Observe that $\Delta_1, \ldots, \Delta_{k} = \emptyset$ for $m=3k$, which implies that $\mathrm{M}(D^n_m) \simeq *$.

    Hence, suppose $m\neq 3k$. Let  $1\leq \nu\leq k-1$ and consider $\Delta_{\nu}$.  For each $1\leq i\leq n$, we match $\bigcup^{k-\nu}_{j=1}(v_{(3j)_i}v_{(3j+1)_i})$ with $(v_{{(3(k-\nu)+2)}_i}v_{{(3(k-\nu)+3)}_i}) \cup \bigcup^{k-\nu}_{j=1}(v_{(3j)_i}v_{(3j+1)_i})$. This produces an acyclic matching for all gradient vector fields in $\Delta_{\nu}$. It remains to put a matching on to $\Delta_k$.

    For $\Delta_k$, we consider cases:
    \begin{description}
    \item[Case 1] Let $m=3k+1$. Then $\Delta_k$ has one element, namely,
    \begin{align*}
    	\bigcup^{n}_{i=1}\bigcup^{k-1}_{j=0}(v_{(2+3j)_i}v_{(3+3j)_i}).
    \end{align*}
    This is an $(nk-1)$-unmatched simplex, so it is critical, and thus $\mathrm{M}(D^n_m) \simeq  \mathbb{S}^{nk-1}$.

    \item[Case 2] Let $m=3k+2$. Then $\Delta_k$ has $2n+1$ elements which are given by
    \begin{eqnarray*}
        &&\bigcup^{n}_{i=1}\bigcup^{k-1}_{j=0}(v_{(2+3j)_i}v_{(3+3j)_i})\\
     &&\text{For each } 1\leq \ell\leq n, (v_{(0)_{\ell}}v_{(1)_{\ell}})\cup \bigcup_{i=1}^n\bigcup^{k-1}_{j=0}(v_{(2+3j)_i}v_{(3+3j)_i})\\
        && \text{For each } 1\leq \ell\leq n,(v_{(0)_{\ell}}v_{(m-1)_{\ell}}) \cup \bigcup_{i=1}^n\bigcup^{k-1}_{j=0}(v_{(2+3j)_i}v_{(3+3j)_i})
    \end{eqnarray*}

We can only create one matching, namely, we match
    \begin{align*}
    	 \bigcup^{n}_{i=1}\bigcup^{k-1}_{j=0}(v_{(2+3j)_i}v_{(3+3j)_i}) \text{ with } (v_{(0)_1}v_{{(1)}_1})\cup \bigcup_{i=1}^n\bigcup^{k-1}_{j=0}(v_{(2+3j)_i}v_{(3+3j)_i}).
    \end{align*}
This leaves $2n-1$ $(nk)$-simplices unmatched. Thus, $\mathrm{M}(D^n_m) \simeq \vee^{2n-1}\mathbb{S}^{nk}$.
    \end{description}
\end{proof}

\section{Future directions and potential pursuits}

\begin{openquestion}
    One direction that seems to hold great potential for computing the homotopy type of the Morse complex concerns the relationship between the homotopy type of the Morse complex and the generalized Morse complex. We argued using two elementary collapses that $\mathcal{G}\M(C_t \vee l) \simeq \M(C_t \vee l)$. Because the former is a flag complex, its homotopy type should theoretically be easier to compute.  Consider another example

    
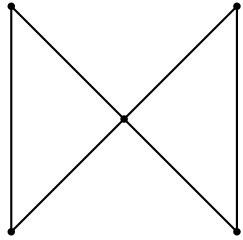
\begin{center}
		\begin{figure} [h]
		
		\begin{tikzpicture}[scale=1.5, node distance = {20mm}, thick, main/.style = {draw, circle}]
		
		    \draw[draw=black] (0,0)--(1,1)--(0,2)--cycle;
		    \draw[draw=black] (2,0)--(1,1)--(2,2)--cycle;
		
		    \node[inner sep=1pt, circle, fill=black](v1) at (0,0) {};
            \node[inner sep=1pt, circle, fill=black](v2) at (0,2) {};

            \node[inner sep=1pt, circle, fill=black](v3) at (1,1) {};

            \node[inner sep=1pt, circle, fill=black](v4) at (2,0) {};
            \node[inner sep=1pt, circle, fill=black](v5) at (2,2) {};
		
		\end{tikzpicture}
		
		\end{figure}
		\end{center}

    Using the Cluster Lemma starting with a star collapse, we can apply a matching to $\M(D^2_3)$, finding that its homotopy type is collapsible. Additionally, using the Cluster Lemma starting with a star cluster collapse, we can apply a matching to $\mathcal{G}\M(D^2_3)$ to compute the homotopy type of a point. We would like to question whether there is a way to use the generalized Morse complex as a tool for computing the homotopy type of the Morse complex for certain complexes.

\end{openquestion}

\begin{openquestion}

One way to use the homotopy type of the generalized Morse complex to determine the homotopy type of the Morse complex is to show that the former collapses to the later. This is in general not always possible since, the homotopy type of the Generalized Morse complex of a cycle computed in Theorem \ref{thm: generalized cycle}, does not agree with the homotopy type of the cycle of the Morse complex. However, one can use the matching found in the proof of Theorem \ref{thm: generalized cycle}, throw out the closed V-paths in the matching, and obtain a matching on the Morse complex.  In this case, the critical cells occur in different dimensions so the homotopy type is not uniquely determined.  However, there may be special cases where the homotopy type can be recovered from knowledge of the critical cells and some other information.  See, for example, \cite[Theorem 2.2]{Kozlov2011}.     In particular, the homotopy type of the Morse complex of the $3-$simplex remains unknown. Chari and Joswig \cite{CJ-2005} showed that the $3-$simplex satisfies $\{b_0 = 1, b_5 = 99\}$ using software.  While we cannot use star clusters to collapse the Morse complex of the $n$-simplex, can we create a matching on the generalized Morse complex and then remove the cyclic gradient vector fields from the matching? Or, can a similar matching strategy provide further insight on how to apply a matching to the Morse complex of the $n-$simplex?









\end{openquestion}

\providecommand{\bysame}{\leavevmode\hbox to3em{\hrulefill}\thinspace}
\providecommand{\MR}{\relax\ifhmode\unskip\space\fi MR }
\providecommand{\MRhref}[2]{%
  \href{http://www.ams.org/mathscinet-getitem?mr=#1}{#2}
}
\providecommand{\href}[2]{#2}

\end{document}